\documentclass[11pt]{amsproc}
\usepackage{fullpage}
\usepackage{amssymb}
\usepackage{amsfonts}
\usepackage[T1]{fontenc}
\usepackage[english]{babel}
\usepackage{amscd,amsmath,amsthm,amssymb,graphics}
\usepackage{lmodern,pst-node}
\usepackage{pstricks}
\usepackage{subcaption}
\usepackage{tikz}

\usetikzlibrary{positioning}
\usepackage{tikz,fullpage}
\usetikzlibrary{arrows,%
	petri,%
	topaths, automata}%
\usepackage{etex}
\usetikzlibrary{decorations.markings}
\tikzstyle{vertex}=[circle, draw, inner sep=0pt, minimum size=6pt]

\usepackage{multicol}
\usepackage{epic,eepic}
\usepackage{amsfonts,amssymb,amscd,amsmath,enumitem,verbatim}
\psset{unit=0.4cm,linewidth=0.8pt,arrowsize=2.5pt 4}
\newpsstyle{fatline}{linewidth=1.5pt}
\newpsstyle{fyp}{fillstyle=solid,fillcolor=verylight}
\definecolor{verylight}{gray}{0.97}
\definecolor{light}{gray}{0.9}
\definecolor{medium}{gray}{0.85}
\definecolor{dark}{gray}{0.6}

\def\frk{\frak}               

\def\Phi{{\frk n}}
\def\Phi{{\frk N}}
\def\opn#1#2{\def#1{\operatorname{#2}}} 
\opn\chara{char} \opn\length{\ell} \opn\pd{pd} \opn\rk{rk}
\opn\projdim{proj\,dim} \opn\injdim{inj\,dim} \opn\rank{rank}
\opn\depth{depth} \opn\grade{grade} \opn\height{height}
\opn\embdim{emb\,dim} \opn\codim{codim}

\opn\Tr{Tr} \opn\bigrank{big\,rank}
\opn\superheight{superheight}\opn\lcm{lcm}
\opn\trdeg{tr\,deg}
	\opn\reg{reg} \opn\lreg{lreg} \opn\ini{in} \opn\lpd{lpd}
	\opn\size{size}\opn\bigsize{bigsize}
	\opn\cosize{cosize}\opn\bigcosize{bigcosize}
	\opn\sdepth{sdepth}\opn\sreg{sreg}
	\opn\link{link}\opn\fdepth{fdepth}
	\opn\deg{deg}
	\opn\max{max}
	\opn\indeg{indeg}
	\opn\min{min}
	\opn\psln{psln}
	\opn\div{div} \opn\Div{Div} \opn\cl{cl} \opn\Cl{Cl}
	\let\epsilon\varepsilon
	\let\phi=\varphi
	\let\kappa=\varkappa

	\opn\Spec{Spec} \opn\Supp{Supp} \opn\supp{supp} \opn\Sing{Sing}
	\opn\Ass{Ass} \opn\Min{Min}\opn\Mon{Mon} \opn\dstab{dstab} \opn\astab{astab}
	\opn\Syz{Syz}
	\opn\Ann{Ann} \opn\Rad{Rad} \opn\Soc{Soc}
	\opn\Im{Im}
	\opn\Ind{Ind}
	\opn\del{del}
	\opn\Ker{Ker} \opn\Coker{Coker} \opn\Am{Am}
	\opn\Hom{Hom} \opn\Tor{Tor} \opn\Ext{Ext} \opn\End{End}
	\opn\Aut{Aut} \opn\id{id}
	
	\opn\nat{nat}
	\opn\pff{pf}
	\opn\Pf{Pf} \opn\GL{GL} \opn\SL{SL} \opn\mod{mod} \opn\ord{ord}
	\opn\Gin{Gin} \opn\Hilb{Hilb}\opn\sort{sort}
	\opn\initial{init}
	\opn\ende{end}
	\opn\height{height}
	\opn\bight{bight}
	\opn\hte{ht}
	\opn\indeg{indeg}
	\opn\reg{reg}
	\opn\depth{depth}
	\opn\type{type}
	\opn\ldim{ldim}
	\opn\maxdeg{maxdeg}
	\opn\aff{aff} \opn\con{conv} \opn\relint{relint} \opn\st{st}
	\opn\lk{lk} \opn\cn{cn} \opn\core{core} \opn\vol{vol}
	\opn\link{link} \opn\star{star}\opn\lex{lex}
	\opn\gr{gr}
	
	\def\pot#1#2{#1[\kern-0.28ex[#2]\kern-0.28ex]}

	\opn\dirlim{\underrightarrow{\lim}}
	\opn\inivlim{\underleftarrow{\lim}}
		\def\Implies{\ifmmode\Longrightarrow \else
			\unskip${}\Longrightarrow{}$\ignorespaces\fi}
		\def\implies{\ifmmode\Rightarrow \else
			\unskip${}\Rightarrow{}$\ignorespaces\fi}
		\def\iff{\ifmmode\Longleftrightarrow \else
			\unskip${}\Longleftrightarrow{}$\ignorespaces\fi}

		\let\:=\colon
		\theoremstyle{plain}
		\newtheorem{Theorem}{Theorem}[section]
		\newtheorem{Lemma}[Theorem]{Lemma}
		\newtheorem{Corollary}[Theorem]{Corollary}
		\newtheorem{Proposition}[Theorem]{Proposition}

		\theoremstyle{definition}
		\newtheorem{Definition}[Theorem]{Definition}
		\newtheorem{Remark}[Theorem]{Remark}
		
		\newtheorem{Example}[Theorem]{Example}

		\let\epsilon\varepsilon
		\let\kappa=\varkappa
		\def\qed{\ifhmode\textqed\fi
			\ifmmode\ifinner\quad\qedsymbol\else\dispqed\fi\fi}
		\def\textqed{\unskip\nobreak\penalty50
			\hskip2em\hbox{}\nobreak\hfil\qedsymbol
			\parfillskip=0pt \finalhyphendemerits=0}
		\def\dispqed{\rlap{\qquad\qedsymbol}}
		
		\opn\dis{dis}
		\def\pnt{{\raise0.5mm\hbox{\large\bf.}}}
		
		\opn\Lex{Lex}
		
\begin{document}
 \title{Unmixed polymatroidal ideals}
\author{Mozghan Koolani, Amir Mafi and Hero Saremi*}
\dedicatory{Dedicated to the memory of our great friend Prof. J\"urgen Herzog}

\address{Mozghan Koolani, Department of Mathematics, University of Kurdistan, P.O. Box: 416, Sanandaj,
Iran.}
\email{mozhgankoolani@gmail.com}

\address{Amir Mafi, Department of Mathematics, University of Kurdistan, P.O. Box: 416, Sanandaj,
Iran.}
\email{a\_mafi@ipm.ir}

\address{Hero Saremi, Department of Mathematics, Sanandaj Branch, Islamic Azad University, Sanandaj, Iran.}

\email{hero.saremi@gmail.com}

\keywords{Hypergraphs, monomial ideals, polymatroidal ideals, unmixedness.\\
* Corresponding author}

\begin{abstract}
Let $R=K[x_1,\ldots,x_n]$ denote the polynomial ring in $n$ variables over a field $K$ and $I$ be a polymatroidal ideal of $R$. In this paper, we provide a comprehensive classification of all unmixed polymatroidal ideals. This work addresses a question raised by Herzog and Hibi in \cite{HH2}.

\end{abstract}

\maketitle

\section{Introduction}  
	
	Let $ R = K[x_1, \ldots, x_n] $ be a polynomial ring in $ n $ variables over a field $ K $. Given a monomial ideal $ I \subseteq R $, we denote by $ G(I) $ the unique minimal set of monomial generators of $ I $. Additionally, we define the set of associated prime ideals of $ R/I $ as $ \Ass(I) $.  
	
	A monomial ideal $ I $ generated in a single degree is termed \textit{polymatroidal} if it fulfills the following exchange condition: for any two elements $ u, v \in G(I) $ with $ \deg_{x_i}(v) < \deg_{x_i}(u) $, there exists an index $ j $ such that $ \deg_{x_j}(u) < \deg_{x_j}(v) $ and $ x_j  ( {u}/{x_i}) \in G(I) $. As noted in \cite{HH1}, such ideals are referred to as polymatroidal because the monomials in the ideal correspond to the bases of a discrete polymatroid.  
	
	Moreover, a polymatroidal ideal $ I $ is classified as \textit{matroidal} if it is generated by square-free monomials. An illustrative example of a polymatroidal ideal is the ideal of Veronese type. Given fixed positive integers $ d $ and $ 1 \leq a_1 \leq \ldots \leq a_n \leq d $, the ideal of Veronese type indexed by $ d $ and $ (a_1, \ldots, a_n) $, denoted $ I_{(d; a_1, \ldots, a_n) }$, is generated by monomials $ u = x_1^{b_1} \cdots x_n^{b_n} $ of $ R $ that have degree $ d $ and satisfy $ b_j \leq a_j $ for all $ 1 \leq j \leq n $.  
	
	Polymatroidal ideals exhibit several notable properties:  
	
	\begin{itemize}  
		\item [(i)] The product of two polymatroidal ideals is also polymatroidal (see \cite{CH}, Theorem 5.3). Consequently, every power of a polymatroidal ideal remains polymatroidal.  
		\item [(ii)] An ideal $ I $ is polymatroidal if and only if $ (I : u) $ is a polymatroidal ideal for all monomials $ u $ (see \cite{BH}, Theorem 1.1). In particular, for every variable $ x_i $, the ideal $ (I : x_i) $ is a polymatroidal ideal of degree $ d - 1 $, when $ I $ is a polymatroidal ideal of degree $ d $.  
	\end{itemize}  
	
	In recent years, numerous authors have concentrated on exploring the properties of polymatroidal ideals. For more comprehensive discussions, refer to \cite{HH1, CH, HH2, HHV, C, HRV, V1, BH, HV, HV1, HQ, BJ, KM1, JMS, SM, MN, HMS, KMS}.

	Herzog and Hibi \cite{HH2} established that a polymatroidal ideal $I$ is Cohen-Macaulay (i.e. CM) if and only if it is one of the following:  
	\begin{itemize}  
		\item A principal ideal,  
		\item A Veronese ideal,  
		\item A squarefree Veronese type ideal.  
	\end{itemize}  
	Additionally, it is important to note that $I$ is CM whenever the quotient ring $ R/I $ is CM as well. They also posed an intriguing question: {\it from a combinatorial perspective, it would be highly valuable to classify all unmixed polymatroidal ideals.} This classification could provide deeper insights into the combinatorial structures inherent to these ideals and their applications.

Recall that an ideal $I$ is termed {\it unmixed} if all prime ideals in the associated primes of $ \text{Ass}(I) $ share the same height. It is well established that every Cohen-Macaulay (CM) ideal is unmixed. Vladoiu, in \cite[Theorem 3.4]{V1}, demonstrated that a Veronese-type ideal $ I $ is unmixed if and only if it is CM. Furthermore, Chiang-Hsieh, in \cite[Theorem 3.4]{C}, showed that if $ I $ is an unmixed matroidal ideal of degree $d$, then the following inequalities hold:  
	${n}/{d} \leq \text{height}(I) \leq n - d + 1.$  In particular, it follows that $ \text{height}(I) = n - d + 1 $ if and only if $ I $ is a squarefree Veronese ideal, and $ \text{height}(I) = \frac{n}{d} $ if and only if $ I = J_1 J_2 \cdots J_d $, where each $ J_i $ is generated by $n/d$ distinct variables, and $ \text{supp}(J_i) \cap \text{supp}(J_j) = \emptyset $ for all $ i \neq j $.  
	
	Let us define the support of an ideal: if $ G(I) = \{u_1, \ldots, u_t\} $, then we set   
	$	\text{supp}(I) = \bigcup_{i=1}^{t} \text{supp}(u_i),$  
	where $ \text{supp}(u) = \{x_i : u = x_1^{a_1} \cdots x_n^{a_n}, \, a_i \neq 0\} $.  
	
Bandari and Jafari, in \cite{BJ}, investigated the class of equidimensional polymatroidal ideals. Specifically, they proved in \cite[Theorem 3.9]{BJ} that an unmixed polymatroidal ideal is connected in codimension one if and only if it is CM. Additionally, the second and third authors, in \cite[Theorem 1.5]{SM}, proved that if $ I $ is a matroidal ideal of degree $ d $, then $ I $ is unmixed if and only if $ (I : x_i) $ is unmixed and $ \text{height}(I) = \text{height}(I : x_i) $ for all $ 1 \leq i \leq n $.  
	
	The primary objective of this paper is to classify all unmixed polymatroidal ideals, a question first posed by Herzog and Hibi in \cite{HH2}. We present the following results:  
	
{\bf Theorem 1:}  
Let $I$ be a matroidal ideal of degree $ d $. Then $ I $ is unmixed if and only if it is the edge ideal of a complete $d$-uniform $m$-partite hypergraph that is $ k $-balanced for some integers $ m, k \geq 1 $.  
		
{\bf Theorem 2:}
		A polymatroidal ideal $ I $ of degree $ d $ is unmixed if and only if one of the following conditions is satisfied:
		\begin{enumerate}
			\item  $I=\frak{m}^d$.
			\item  $I=\frak{p}_1^{a_1}\frak{p}_2^{a_2}\ldots\frak{p}_t^{a_t}$, where $\frak{p}_i$'s are prime ideals with $\height(\frak{p}_i)=\height(\frak{p}_j)$ and $G(\frak{p}_i)\cap G(\frak{p}_j)=\emptyset$ for all $1\leq i\neq j\leq t$ and $\sum_{i=1}^ta_i=d$.
			\item $I=\frak{p}_1^{a_1}\frak{p}_2^{a_2}\ldots\frak{p}_t^{a_t}{J}$, where $\frak{p}_i$'s are prime ideals and $J$ is an unmixed matroidal ideal such that $\height(\frak{p}_i)=\height(\frak{p}_j)=\height(J)$, $G(\frak{p}_i)\cap G(\frak{p}_j)=\emptyset$, $G(\frak{p}_i)\cap G(J)=\emptyset$ for all $1\leq i\neq j\leq t$ and $\sum_{i=1}^ta_i +\deg(J)=d$. 
			\item $I$ is an unmixed matroidal ideal of degree $d$.
			
		\end{enumerate} 
  
For any concepts or terminology that have not been explained, we direct the reader to \cite{HH3} and \cite{V}. Additionally, several explicit examples were generated with the assistance of the computer algebra system Macaulay2 \cite{GS}.

\section{The results}

In this section, we assume that all polymatroidal ideals are fully supported, meaning that for every polymatroidal ideal $I$, the support satisfies $ \supp(I) = \{x_1, \ldots, x_n\}=[n]$. We begin with the following straightforward lemma.

\begin{Lemma}\label{L0}
Let $I$ be a matroidal ideal of degree $d$ and $x,y$ be two variables in $R$. Then $xy\nmid u$ for all $u\in G(I)$ if and only if $(I:x)=(I:y)$. 
\end{Lemma}

\begin{proof}
For the first direction, we have $(I:x)=(I:xy)$ and $(I:y)=(I:yx)$ and from this we have $(I:x)=(I:y)$. The converse is clear.
\end{proof}

\begin{Proposition}\label{P0}
Let $I$ be a matroidal ideal of degree $d\geq 2$ and $x,y$ be two variables in $R$. Then there are subsets $S_1,\ldots, S_m$ of $[n]$ such that the following conditions hold:
\begin{itemize}
\item[(i)] $m\geq d$ and $\mid S_i\mid\geq 1$ for all $i$;
\item[(ii)] $S_i\cap S_j=\emptyset$ for all $1\leq i\neq j\leq m$ and $\bigcup_{i=1}^m S_i=[n]$;
\item[(iii)] $xy\mid u$ for some $u\in G(I)$ if and only if $x\in S_i$ and $y\in S_j$ for $1\leq i\neq j\leq m$;
\item[(iv)] $xy\nmid u$ for all $u\in G(I)$ if and only if $x,y\in S_i$ for some $i$.
\end{itemize}
\end{Proposition}

\begin{proof}
$(i)$ Since $I$ is a squarefree monomial ideal, it follows that $\depth R/I>0$ and this implies that $(I:\frak{m})=I$.
Therefore, there exists $m\leq n$ such that $I=(I:\frak{m})=\bigcap_{i=1}^m(I:x_i)$ is a minimal intersection of $I$. Set $S_i=[n]\setminus\supp(I:x_i)$ for all $1\leq i\leq m$. It is clear that $\mid S_i\mid\geq 1$ for all $i$. Now, suppose $u=x_1x_2\ldots x_d$ is an element of $G(I)$. By Lemma \ref{L0}, it follows that $(I:x_i)\neq(I:x_j)$ for all $1\leq i\neq j\leq d$ and so  $m\geq d$.\\
$(ii)$ Suppose $y\in S_i\cap S_j$ for $1\leq i\neq j\leq m$. Then $y\notin\supp(I:x_i)\cup\supp(I:x_j)$ and hence $yx_i , yx_j\nmid u$ for all $u\in G(I)$. Therefore by Lemma \ref{L0}, we conclude that $(I:x_i)=(I:y)=(I:x_j)$ and this is a contradiction. Thus $S_i\cap S_j=\emptyset$ for all $1\leq i\neq j\leq m$ and also it is clear that $\bigcup_{i=1}^m S_i=[n]$.\\
$(iii)$ If $xy\mid u$ for some $u\in G(I)$, then by Lemma \ref{L0}, $(I:x)\neq(I:y)$. Since $y\notin\supp(I:y)$  and $x\notin\supp(I:x)$, it follows that $x\in S_i=[n]\setminus\supp(I:x)$ and $y\in S_j=[n]\setminus\supp(I:y)$ for $1\leq i\neq j\leq m$. Conversely, suppose $x\in S_i$ and $y\in S_j$ for 
 $1\leq i\neq j\leq m$, $S_i=[n]\setminus\supp(I:x_i)$ for all $1\leq i\leq m$. Then $x\notin\supp(I:x_i)$, $y\notin\supp(I:x_j)$ 
and so $xx_i\nmid u, yx_j\nmid u$ for all $u\in G(I)$. Thus by Lemma \ref{L0}, it follows that $(I:x)=(I:x_i)$ and $(I:y)=(I:x_j)$. If $xy\nmid u$ for all $u\in G(I)$, then $(I:x_i)=(I:x)=(I:y)=(I:x_j)$ and so $S_i=S_j$ and this is a contradiction. Hence $xy\mid u$ for some $u\in G(I)$.\\
$(iv)$ If $x,y\in S_i$ for some $i$, then $x,y\notin\supp(I:x_i)$ and so $yx_i,xx_i\nmid u$ for all $u\in G(I)$. Hence by Lemma \ref{L0}, $(I:x)=(I:x_i)=(I:y)$ and so $xy\nmid u$ for all $u\in G(I)$. Conversely, if $xy\nmid u$ for all $u\in G(I)$, then $(I:x)=(I:y)$. Thus $x,y\in S_i$ for some $i$. This completes the proof.
\end{proof}

The following result is a direct consequence of Proposition \ref{P0}.

\begin{Corollary}
Let $I$ be a matroidal ideal of degree $d\geq 2$. Then either\\ $\supp(I:x_i)=\supp(I:x_j)$ or $\supp(I:x_i)\cup\supp(I:x_j)=[n]$
for all $1\leq i, j\leq n.$
\end{Corollary}

\begin{Proposition}\label{P1}
Let $I$ be an unmixed matroidal ideal of degree $3$. Then\\ $\mid\supp(I:x)\mid=\mid\supp(I:y)\mid$ for all $x,y\in [n]$.
\end{Proposition}

\begin{proof}
Since $I$ is a squarefree monomial ideal, as discussed previously, we may assume that $I=\bigcap_{i=1}^m(I:x_i)$ $(\dagger)$ is a minimal intersection of $I$. Suppose $x,y\in[n]$. If $(I:x)=(I:y)$, then there is nothing to prove. Now, suppose that $(I:x)\neq(I:y)$. Then $(I:x)$ and $(I:y)$ are in the minimal intersection $(\dagger)$ of $I$. Suppose $I=\frak{p}_1\cap\frak{p}_2\cap\ldots\cap\frak{p}_r$ is a minimal primary decomposition of $I$. Since $I$ is an unmixed matroidal ideal of degree $3$ and a minimal generator of $I$ has the form $\lcm\{x_{j_1},\ldots, x_{j_r}\}$, where $x_{j_i}$ is an element of $G(\frak{p}_i)$ for $1\leq i\leq r$, the number of $\frak{p}_i$'s in the minimal primary decomposition of $I$ in which every variable such as $x$ is not an element of $G(\frak{p}_i)$ are precisely $2$, otherwise $G(I)$ will has an element of degree $>3$ or an element of degree $<3$ and this is impossible. It therefore follows $\mid\Ass(I:x)\mid=\mid\Ass(I:y)\mid$. Now, set $S_i^x=\supp(I:x)\setminus G(I:xx_i)$ and $S_i^y=\supp(I:y)\setminus G(I:yx_i)$. Thus, by Proposition \ref{P0}, there exist integers $t,s\geq 1$ such that $\bigcup_{i=1}^tS_i^x=\supp(I:x)$ and $\bigcup_{i=1}^sS_i^y=\supp(I:y)$. Since $\mid\Ass(I:x)\mid=\mid\Ass(I:y)\mid$ and for each $i$, $G(I:xx_i)$ and $G(I:yx_i)$ are the minimal generator set of the associated prime of $\Ass(I:x)$ and $\Ass(I:y)$ respectively, this implies that $t=s$. By using \cite[Corollary 1.3]{SM}, we have $t(\mid\supp(I:x)\mid\setminus \height(I:x))=\mid\supp(I:x)\mid$ and  $t(\mid\supp(I:y)\mid\setminus \height(I:y))=\mid\supp(I:y)\mid$. Since $\height(I:x)=\height(I:y)=\height(I)$ and $t>1$, it immediately deduces that $\mid\supp(I:x)\mid=\mid\supp(I:y)\mid$, as required. 
\end{proof}

\begin{Theorem}\label{T0}
Let $I$ be an unmixed matroidal ideal of degree $d\geq 2$. Then $\mid\supp(I:x)\mid=\mid\supp(I:y)\mid$ for all $x,y\in [n]$.
\end{Theorem}

\begin{proof}
We use induction on $d$. If $d=2$, then $\supp(I:x)=G(I:x)$ and $\height(I)=\height(I:x)=\mid G(I:x)\mid$ for all $x\in [n]$. Hence the result holds in this case. Now, consider $d\geq 3$ and the result has been established for $d-1$. Given that $(I:x)$ and $(I:y)$ are squarefree monomial ideals, and $\depth(R/(I:x), \depth(R/(I:y)\geq 1$, it follows that $((I:x):\frak{m})=(I:x)$ and $((I:y):\frak{m})=(I:y)$. Consequently, we can express  $(I:x)=\bigcap_{i=1}^n(I:xx_i)$ and $(I:y)=\bigcap_{i=1}^n(I:yx_i)$. If $y\notin\supp(I:x)$, then by applying Lemma \ref{L0}, we find that $(I:x)=(I:y)$ and the result follows in this case. Therefore, we assume that $y\in\supp(I:x)$ and $x\in\supp(I:y)$, along with the conditions $x=x_l$ and $y=x_k$. This leads to the identities  $(I:x)=(I:xy)\cap\bigcap_{l\neq i=1}^n(I:xx_i)$ and 
$(I:y)=(I:yx)\cap\bigcap_{k\neq i=1}^n(I:yx_i)$. Since $(I:x)$ and $(I:y)$ are matroidal ideal of degree $d-1$, by the induction hypothesis we have $\mid\supp(I:xx_i)\mid=\mid\supp(I:xy)\mid=\mid\supp(I:yx_i)\mid$ for all $i$. Now, we may assume that $I=x_1I_1+I_2$, where $I_1, I_2$ are squarefree monomial ideals and $x_1\notin\supp(I_2)\cup\supp(I_1)$. By \cite[Theorem 1.1]{BH}, $(I:x_1)$ is matroidal ideal of degree $d-1$ and indeed $(I:x_1)=I_1$.  By exchange property we immediately conclude that $I_2\subseteq I_1$ is also matroidal. Without losing the generality, we may assume that $\supp(I:x_1)=\{x_{r+1},x_{r+2},\ldots,x_n\}$. Furthermore, we may consider that $I_2=x_2I_3+I_4$, where $I_3,I_4$ are matroidal ideal with $I_4\subseteq I_3$ and $x_2\notin\supp(I_3)\cup\supp(I_4)$. Hence $I=x_1I_1+I_2=x_1I_1+x_2I_3+I_4$. Since $x_2\notin\supp(I:x_1)$, we find that $(I:x_1)=(I:x_2)$ and $(I:x_2)=I_3$. Therefore, it follows that $I=(x_1,x_2)(I:x_1)+I_4$. By continuing this processes we conclude that $I=(x_1,x_2,\ldots,x_r)(I:x_1)+J$, where $J$ is a matroidal ideal with $J\subseteq (I:x_1)$ and $\supp(J)\subseteq\supp(I:x_1)$. Hence $(I:x_{r+1})=(x_1,x_2,\ldots,x_r)(I:x_1x_{r+1})+(J:x_{r+1})$. Since $(J:x_{r+1})=(J:x_1x_{r+1})$ it is evident that $\supp(J:x_{r+1})\subseteq\supp(I:x_1x_{r+1})$. Thus $\supp(I:x_{r+1})=\{x_1,x_2,\ldots,x_r\}\bigcup\supp(I:x_1x_{r+1})$. By applying induction hypothesis, we can conclude that $\mid\supp(I:x_1x_i)\mid=\mid\supp(I:x_1x_j)\mid$. Hence, we derive that $\mid\supp(I:x_i)\mid=\mid\supp(I:x_j)\mid$ for all $x_i,x_j\in\supp(I:x_1)$. Since $x_1,x_{r+1}\in\supp(I:x_r)$, by using the similar arguments as above it follows that $\mid\supp(I:x_1)\mid=\mid\supp(I:x_{r+1})\mid$. Therefore $\mid\supp(I:x_1)\mid=\mid\supp(I:x_{i})\mid$ for all $i=r+1,\ldots, n$. Since $(I:x_1)=(I:x_j)$ for all $j=1,\ldots, r$, it therefore follows that $\mid\supp(I:x)\mid=\mid\supp(I:y)\mid$ for all $x,y\in [n]$, as required.
\end{proof}

Following \cite{KM}, a hypergraph $\mathcal{H}$ with finite vertex set $V(\mathcal{H})=[n]$ is a collection of nonempety subsets of $[n]$ whose union is $[n]$, called edges. A hypergraph $\mathcal{H}$ is said $d$-uniform if all its edges have cardinality $d$. A $d$-uniform hypergraph $\mathcal{H}$ is said to be $m$-partite if its vertex set $[n]$ can be partitioned into sets $S_1,S_2,\ldots,S_m$, such that every edge in the edge set $E(\mathcal{H})$ contains at most one vertex from each $S_i$. The complete $d$-uniform $m$-partite hypergraph consists all possible edges satisfying this condition. A $m$-partite hypergraph is called $k$-balanced if $\mid S_i\mid=k$ for all $i=1,\ldots,m$. For more details about multipartite hypergraph, see also \cite{E}.

\begin{Theorem}\label{T1}
Let $I$ be a matroidal ideal of degree $d$. Then $I$ is unmixed if and only if $I$ is an edge ideal of a complete $d$-uniform $m$-partite hypergraph which is $k$-balanced for some $m, k\geq 1$.
\end{Theorem}

\begin{proof}
$(\Rightarrow)$. We may assume that $d\geq 2$. By using Proposition \ref{P0}, there exists integer $m\geq 1$ such that for all $1\leq i\leq m$ we may consider $S_i=[n]\setminus\supp(I:x_i)$. Then  $S_i\cap S_j=\emptyset$ for all $1\leq i\neq j\leq m$ and $\bigcup_{i=1}^m S_i=[n]$. By Theorem \ref{T0}, $\mid\supp(I:x_i)\mid=\mid\supp(I:x_j)\mid$ for all $1\leq i,j\leq n$ and this implies that 
$\mid S_i\mid=\mid S_j\mid=k$ for all $1\leq i\neq j\leq m$ and some $k\geq 1$. Thus $I$ is an edge ideal of a complete $d$-uniform $m$-partite hypergraph which is $k$-balanced for some $m, k\geq 1$.\\
$(\Leftarrow)$. It follows by \cite[Thorem 3.1]{KM}. 
\end{proof}

\begin{Remark}\label{R0}
From Proposition \ref{P0} and Theorem \ref{T0}, we conclude that if $I$ is an unmixed matroidal ideal of degree $d$, then $S_i$'s are uniquely determined and so $\mid S_i\mid=\mid S_j\mid=k$ and in this case $mk=m\mid S_i\mid=n$, where $S_i=[n]\setminus\supp(I:x_i)$ and $\bigcup_{i=1}^mS_i=[n]$. Furthermore, by \cite[Proposition 2.1(e)]{KM}, $\height(I)=k(m-d+1)=n-k(d-1)$. In particular, $m=n$ if and only if $\mid S_i\mid=1$ and so $\height(I)=n-d+1$. Therefore $I$ is a squarefree Veronese type. Moreover, if $n$ is a prime integer number, then since $m\geq d$ it follows $\mid S_i\mid=1$ and again in this case $I$ is a squarefree Veronese type.
\end{Remark}

\begin{Corollary}\label{C0}
Let $I$ be an unmixed matroidal ideal of degree $d$. Then $I$ is squarefree Veronese type if and only if $\mid\supp(I:x_i)\mid=n-1$ for some $1\leq i\leq n$.
\end{Corollary}

\begin{proof}
If $I$ is squarefree Veronese type, then by \cite[Lemma 5.1(b)]{HRV} $\mid\supp(I:x_i)\mid=n-1$ for all $i$. Conversely, if $\mid\supp(I:x_i)\mid=n-1$, then $\mid S_i\mid=1$ and hence by Proposition \ref{P0} and Theorem \ref{T0}, $\mid S_i\mid=1$ for all $1\leq i\leq n$. Therefore $I$ is squarefree Veronese type.
\end{proof} 

\begin{Corollary}
Let $I$ be a matroidal ideal of degree $d$. Then $I$ is squarefree Veronese type if and only if $\supp(I:x_i)\cup\supp(I:x_j)=[n]$ for all $1\leq i\neq j\leq n$.
\end{Corollary}

\begin{proof}
If $I$ is squarefree Veronese type, then the result is clear. Conversely, if $\supp(I:x_i)\cup\supp(I:x_j)=[n]$ for all $1\leq i\neq j\leq n$, then by applying Proposition \ref{P0}, we have $n=m$. Thus $\mid S_i\mid=1$ and so  $\mid\supp(I:x_i)\mid=n-1$  for all $1\leq i\leq n$. Now, by using Theorem \ref{T1} and Corollary \ref{C0} we immediately conclude that $I$ is squarefree Veronese type. 
\end{proof} 

To clarify the preceding results, we present several examples below.
\begin{Example}
Consider $n=6$ and $I$ is an ideal of degree $3$ in $R$ with  $G(I)=\{x_1x_3x_5,x_1x_3x_6,x_1x_4x_5,x_1x_4x_6,x_2x_3x_5,x_2x_3x_6,x_2x_4x_5,x_2x_4x_6\}$.
Then $I$ is unmixed matroidal. Also, $S_1=\{x_1,x_2\}$, $S_2=\{x_3,x_4\}, S_3=\{x_5,x_6\}$ and  $I$ is an edge ideal of a complete $3$-uniform $3$-partite hypergraph which is $2$-balanced and by Remark \ref{R0}, $\height(I)=2(3-3+1)=2$.
\end{Example}

\begin{Example}
Let $I$ be an ideal of degree $3$ in $R$ with $n=9$ and\\ $G(I)=\{x_1x_4x_7,x_1x_4x_8,x_1x_4x_9,x_1x_5x_7,x_1x_5x_8,x_1x_5x_9,x_1x_6x_7,x_1x_6x_8,x_1x_6x_9,
x_2x_4x_7,\\x_2x_4x_8,x_2x_4x_9,x_2x_5x_7,x_2x_5x_8,x_2x_5x_9,x_2x_6x_7,x_2x_6x_8,x_2x_6x_9,
x_3x_4x_7,x_3x_4x_8,x_3x_4x_9,\\x_3x_5x_7,x_3x_5x_8,x_3x_5x_9,x_3x_6x_7,x_3x_6x_8,x_3x_6x_9\}.$
Then $I$ is unmixed matroidal  with $S_1=\{x_1,x_2,x_3\}$, $S_2=\{x_4,x_5,x_6\}, S_3=\{x_7,x_8,x_9\}$ and  $I$ is an edge ideal of a complete $3$-uniform $3$-partite hypergraph which is $3$-balanced and so $\height(I)=3(3-3+1)=3.$
\end{Example}

\begin{Example}\cite{HH2}
Consider the unmixed matroidal ideal $I$ of degree $2$ with $n=6$ and  $I=(x_1x_3,x_1x_4,x_1x_5,x_1x_6,x_2x_3,x_2x_4,x_2x_5,x_2x_6,x_3x_5,x_3x_6,x_4x_5,x_4x_6)$. Then 
$S_1=\{x_1,x_2\}$, $S_2=\{x_3,x_4\}, S_3=\{x_5,x_6\}$ and  $I$ is an edge ideal of a complete $2$-uniform $3$-partite hypergraph which is $2$-balanced and $\height(I)=2(3-2+1)=4$.
\end{Example}

For our subsequent results, we employ the following established theorem.

\begin{Theorem}\cite[Theorems 3, 5]{HV}\label{T2}
Let $I$ be a polymatroidal ideal of $R$ with $\Ass(I)=\{\frak{p}_1,\ldots,\frak{p}_t\}$. Then there are integers $a_i\geq 0$ such that $I=\bigcap_{i=1}^t\frak{p}_i^{a_i}$. 
\end{Theorem}

\begin{Lemma}\label{L1}
Let $I$ be a polymatroidal ideal of degree $2$. $I$ is unmixed if and only if $I$ is an unmixed matroidal ideal or $I=\frak{m}^2$.
\end{Lemma}

\begin{proof}
$(\Rightarrow)$. If $\frak{m}\in\Ass(I)$, then $\Ass(I)=\{\frak{m}\}$ and by Theorem \ref{T2}, it follows that $I=\frak{m}^2$. If $\frak{m}\notin\Ass(I)$, then $I=(I:\frak{m})=\bigcap_{i=1}^n(I:x_i)$. Since all $(I:x_i)$ are polymatroidal ideals of degree $1$, it follows that $I$ is a squarefree unmixed polymatroidal ideal and so $I$ is an unmixed matroidal ideal.\\
$(\Leftarrow)$. It is clear.
\end{proof}

\begin{Theorem}\label{T3}
A polymatroidal ideal $I$ of degree $d$ is unmixed if and only if one of the following statements holds:
\begin{itemize}
\item[(i)] $I=\frak{m}^d$.
\item[(ii)] $I=\frak{p}_1^{a_1}\frak{p}_2^{a_2}\ldots\frak{p}_t^{a_t}$, where $\frak{p}_i$'s are prime ideals with $\height(\frak{p}_i)=\height(\frak{p}_j)$ and $G(\frak{p}_i)\cap G(\frak{p}_j)=\emptyset$ for all $1\leq i\neq j\leq t$ and $\sum_{i=1}^ta_i=d$.
\item[(iii)] $I=\frak{p}_1^{a_1}\frak{p}_2^{a_2}\ldots\frak{p}_t^{a_t}{J}$, where $\frak{p}_i$'s are prime ideals and $J$ is an unmixed matroidal ideal such that $\height(\frak{p}_i)=\height(\frak{p}_j)=\height(J)$, $G(\frak{p}_i)\cap G(\frak{p}_j)=\emptyset$, $G(\frak{p}_i)\cap G(J)=\emptyset$ for all $1\leq i\neq j\leq t$
 and $\sum_{i=1}^ta_i +\deg(J)=d$. 
\item[(iv)]  $I$ is an unmixed matroidal ideal of degree $d$.
\end{itemize}
\end{Theorem}

\begin{proof}
$(\Leftarrow)$. It is clear.\\
$(\Rightarrow)$. Let $I$ be an unmixed polymatroidal ideal. If $\frak{m}\in\Ass(I)$, then $\Ass(I)=\{\frak{m}\}$ and by applying Theorem \ref{T2}, we have $I=\frak{m}^d$. Now, suppose that $\frak{m}\notin\Ass(I)$ and in this case $\mid\Ass(I)\mid\geq 2$. We prove the result by induction on $d$. If $d=2$, then by Lemma \ref{L1}, we have the result. Suppose $d=3$. Since $\frak{m}\notin\Ass(I)$, we have 
$I=(I:\frak{m})=\bigcap_{i=1}^n(I:x_i)$. Since all $(I:x_i)$ are polymatroidal ideals of degree $2$, by Lemma \ref{L1}, we have the following minimial primary decomposition $I=\frak{p}_1^2\cap\frak{p}_2^2\cap\ldots\cap\frak{p}_t^2\cap\frak{q_1}\cap\ldots\cap\frak{q}_s$, where $t,s$ are non-negative integers. If $t\geq 2$, then there exists $x_i\in\frak{p}_2\setminus\frak{p}_1$, since $\frak{p}_2\nsubseteq\frak{p}_1$. Therefore, by using \cite[Lemma 2.1]{M}, $\frak{p}_1^2\cap\frak{p}_2$ is a factor member in the minimal primary decomposition of $(I:x_i)$ and this is impossible by Lemma \ref{L1}. Thus $t\leq 1$. If $t=0$, then $I$ is an unmixed matroidal ideal.
Let $t=1$. If $s\geq 2$, then $\frak{q}_2\nsubseteq\frak{p}^2_1\cap\frak{q_1}$, since $I$ is unmixed. In this case, there exists  $x_i\in\frak{q}_2\setminus\frak{p}_1^2\cap\frak{q}_1$ and again by using \cite[Lemma 2.1]{M}, $\frak{p}_1^2\cap\frak{q}_1$  is a factor member in the minimal primary decomposition of $(I:x_i)$ and this is impossible by Lemma \ref{L1}. Thus $I=\frak{p}_1^2\cap\frak{q}_1$. If $x_i\in\frak{p}_1\cap\frak{q}_1$, then $x_i^2\in I$ and this is impossible, since $I$ is of degree $3$. Therefore $G(\frak{p}_1)\cap G(\frak{q}_1)=\emptyset$ and $I=\frak{p}^2_1\frak{q}_1$. 
Suppose $d\geq 4$ and the result has been proved for $d-1$. By applying Theorem \ref{T2} and our assumption, there are integers $a_i>0$ such that $I=\frak{p}_1^{a_1}\cap\frak{p}_2^{a_2}\cap\ldots\cap\frak{p}_s^{a_s}\cap\ldots\cap\frak{p}_t^{a_t}$ such that all $\frak{p}_i$'s have the same height. If $a_i=1$ for each $1\leq i\leq t$, then $I$ is an unmixed matroidal ideal. Suppose  $a_i\geq 2$ for $1\leq i\leq s$ and $a_j=1$ for $s+1\leq j\leq t$ and we may assume that $s\geq 1$. If $G(\frak{p}_i)\cap G(\frak{p}_j)=\emptyset$ for all $1\leq i\leq t$, then $I=\frak{p}_1^{a_1}\frak{p}_2^{a_2}\ldots\frak{p}_t^{a_t}$ and 
$\sum_{i=1}^ta_i=d$. Now, suppose that $G(\frak{p}_i)\cap G(\frak{p}_j)\neq\emptyset$ for some $1\leq i\neq j\leq t$. If $t=s$, then $a_i\geq 2$ for $1\leq i\leq t$
and also $t\geq 2$ since $\frak{m}\notin\Ass(I)$. From $G(\frak{p}_i)\cap G(\frak{p}_j)\neq\emptyset$, we can choose $x_l\in G(\frak{p}_i)\cap G(\frak{p}_j)$.
Since $\frak{p}_i^{a_i-1}\cap \frak{p}_j^{a_j-1}$ is a factor member in the minimal primary decomposition of $(I:x_l)$ and $(I:x_l)$ is a polymatroidal ideal of degree $d-1$, by induction hypothesis we deduce that $G(\frak{p}_i)\cap G(\frak{p}_j)=\emptyset$ and this is contradiction. Therefore, in this case $G(\frak{p}_i)\cap G(\frak{p}_j)=\emptyset$ for all $1\leq i\neq j\leq t$ and so $I$ satisfies in condition $(ii)$. Now, we assume that $1\leq s<t$. 
By using the above argument, we can assume that $G(\frak{p}_i)\cap G(\frak{p}_j)=\emptyset$ for all $1\leq i\neq j\leq s$. Also, if $s=1$ and $t=s+1$ then again by the above argument we have  $G(\frak{p}_i)\cap G(\frak{p}_j)=\emptyset$ for all $1\leq i\neq j\leq t$ and $I$ satisfies in condition $(ii)$. Now, suppose that 
$s+1<t$, by the above mention $G(\frak{p}_i)\cap G(\frak{p}_j)=\emptyset$ for all $1\leq i\neq j\leq s$. Let $1\leq i\leq s$ and $1+s\leq j\leq t$ such that 
$G(\frak{p}_i)\cap G(\frak{p}_j)\neq\emptyset$. We can choose a variable $x_r$ from out of $G(\frak{p}_i)\cup G(\frak{p}_j)$. Then $(I:x_r)$ is a polymatroidal ideal of degree $d-1$ with $G(\frak{p}_i)\cap G(\frak{p}_j)\neq\emptyset$ and this is a contradiction. Hence $G(\frak{p}_i)\cap G(\frak{p}_j)=\emptyset$ for all $1\leq i\neq j\leq s$ and $G(\frak{p}_r)\cap G(\frak{p}_l)=\emptyset$  for all $1\leq r\leq s$ and all $1+s\leq l\leq t$. Therefore
from $I=\frak{p}_1^{a_1}\cap\frak{p}_2^{a_2}\cap\ldots\cap\frak{p}_s^{a_s}\cap\ldots\cap\frak{p}_t^{a_t}$, we set $J=\frak{p}_j^{a_j}\cap\ldots\cap\frak{p}_l^{a_l}$ such that $s+1\leq j\leq l\leq t$ and for each $r=j,\ldots, l$ there is $j\leq k\neq r\leq l$ such that 
$G(\frak{p}_r)\cap G(\frak{p}_k)\neq\emptyset$. It is clear that there is a monomial element $u$ of $R$ such that $(I:u)=J$ and $G(\frak{p}_i)\cap G(J)=\emptyset$ for all $i\neq j,\ldots, l$. Thus $I$ satisfy in condition $(iii)$. This completes the result.
\end{proof}

To clarify Theorem \ref{T3}, we present the following examples.
\begin{Example}\label{E1}
Let $n=4$ and $I=(x_1^2x_3,x_1^2x_4,x_2^2x_3,x_2^2x_4,x_1x_2x_3,x_1x_2x_4)$. Then $I$ is unmixed polymatroidal of degree $3$ and by applying Theorem \ref{T2},
$I=(x_1,x_2)^2\cap(x_3,x_4)$. Therefore $I=(x_1,x_2)^2(x_3,x_4)$ as Theorem \ref{T3}(ii).
\end{Example}

\begin{Example}\label{E2}
Let $n=5$ and \[I=(x_1^2x_3x_4,x_1^2x_3x_5,x_1^2x_4x_5,x_2^2x_3x_4,x_2^2x_3x_5,x_2^2x_4x_5,x_1x_2x_3x_4,x_1x_2x_3x_5,x_1x_2x_4x_5).\] Then $I$ is unmixed polymatroidal of degree $4$ and by using Theorem \ref{T2}, $I=(x_1,x_2)^2\cap(x_3,x_4)\cap(x_3,x_5)\cap(x_4,x_5)$. Thus $I=(x_1,x_2)^2(x_3x_4,x_3x_5,x_4x_5)$
 as Theorem \ref{T3}(iii).
\end{Example}

Theorem \ref{T3} immediately implies the following result.
\begin{Corollary}\cite[Theorem 3.4]{V1}
Let $I$ be a Veronese type ideal. Then $I$ is unmixed if and only if $I$ is CM. 
\end{Corollary}

Hartshorne, in \cite[Proposition 1.1]{H}, provides the following definition, which is also referenced in \cite[Definition 3.1]{BJ}.

\begin{Definition}\label{D}
A monomial ideal $I$ of $R$ with height $h$ is connected in codimension one when for each pair of distinct prime ideals $\frak{p},\frak{q}\in\Min(I)$ there exists a sequence of minimial prime ideals $\frak{p}=\frak{p}_1,\ldots,\frak{p}_r=\frak{q}$ such that $\mid G(\frak{p}_i+\frak{p}_{i+1})\mid=h+1$ for all $1\leq i\leq r-1$. In particular, in this case $I$ is equidimentional and $\mid G(\frak{p}_i\cap\frak{p}_{i+1})\mid=h-1$ for all $1\leq i\leq r-1$.
\end{Definition}

We provide a simplified proof of the main results presented in \cite{BJ}.
\begin{Corollary}\cite[Theorem 3.6]{BJ}\label{C1} Let $I$ be a monomial ideal. Then $I$ is a matroidal ideal of connected in codimension one if and only if $I$ is a squarefree Veronese type.
\end{Corollary}

\begin{proof}
Suppose $I$ is a matroidal ideal of connected in codimension one and $I=\bigcap_{i=1}^t\frak{p}_i$ be a minimal primary decomposition of $I$. Then by Definition \ref{D}, we have $x_j\in\supp(I:x_i)$ for all $1\leq i\neq j\leq n$. Thus $\mid\supp(I:x_i)\mid=n-1$ for all $1\leq i\leq n$ and by Corollary \ref{C0}, $I$ is a squarefree Veronese type.
Conversely, every squarefree Veronese type is CM and so we have the result by \cite[Corollary 2.4]{H}.
\end{proof}

\begin{Corollary}\cite[Theorem 3.9]{BJ} Let $I$ be an unmixed polymatroidal ideal. Then $I$ is connected in codimension one if and only if $I$ is CM.
\end{Corollary}

\begin{proof}
If  $I$ is  CM, then the result follows from \cite[Corollary 2.4]{H}. Conversely, if $I$ is an unmixed polymatroidal ideal that is connected in codimension one, then by applying Theorem \ref{T3} and Corollary \ref{C1} we conclude that $I$ is CM.
\end{proof}
\subsection*{Acknowledgements} We are grateful to an anonymous referee for substantial input to improving the article,
and in particular for providing the general statement of Theorem \ref{T3} and examples \ref{E1} and \ref{E2}.



\begin{thebibliography}{}

\bibitem{BH}
S. Bandari and J. Herzog, {\it Monomial localizations and polymatroidal ideals}, Eur. J. Comb.,
{\bf 34}(2013), 752-763.
\bibitem{BJ}
S. Bandari and R. Jafari, {\it On certain equidimensional polymatroidal ideals},  Manuscripta Math., {\bf 149}(2016), 223-233. 
\bibitem{C}
H. J. Chiang-Hsieh, {\it Some arithmetic properties of matroidal ideals}, Comm. Algebra, {\bf 38}(2010), 944-952.
\bibitem{CH}
A. Conca and J. Herzog, {\it Castelnuovo-Mumford regularity of products of ideals}, Collect. Math., {\bf 54}(2003), 137-152.

\bibitem{E}
E. Emtander,  (2009). {\it Betti Numbers of Hypergraphs}, Comm. Algebra, {\bf 37}(2009), 1545-1571. 

\bibitem{GS}
D. R. Grayson and M. E. Stillman, {\it Macaulay 2, a software system for research in algebraic geometry}, Available at
{http://www.math.uiuc.edu/Macaulay2/}.
\bibitem{HMS}P. M. Hamaali, A. Mafi and H. Saremi, {\it A characterization of sequentially Cohen-Macaulay matroidal ideals}, 
Algebra Colloq., {\bf 30}(2023),  237-244. 

\bibitem{H} R. Hartshorne, {\it Complete intersections and connectedness}, Amer. J. Math., {\bf 84}(1962), 497-508. 

\bibitem{HH1} J. Herzog and T. Hibi, {\it Discrete polymatroids}, J. Algebr. Combin., {\bf 16}(2002), 239-268.
\bibitem{HH2}
J. Herzog and T. Hibi, {\it Cohen-Macaulay polymatroidal ideals}, Eur. J. Comb., {\bf 27}(2006), 513-517.
\bibitem{HH3}
J. Herzog and T. Hibi, {\it Monomial ideals}, GTM., {\bf 260}, Springer, Berlin, (2011).

\bibitem{HHV} J. Herzog, T. Hibi and M. Vladoiu, {\it Ideals of fiber type and polymatroids}, Osaka J. Math., {\bf 42}(2005), 807-829.
\bibitem{HQ} J. Herzog and A. Qureshi, {\it Persistence and stability properties of powers of ideals}, J. Pure and Appl. Algebra, {\bf 219}(2015),
    530-542.
\bibitem{HRV} J. Herzog, A. Rauf and M. Vladoiu, {\it The stable set of associated prime ideals of a polymatroidal ideal}, J. Algebr. Comb., {\bf 37}(2013), 289-312.
\bibitem{HV1} J. Herzog and M. Vladoiu, {\it Squarefree monomial ideals with constant depth function}, J. Pure and Appl. Algebra, {\bf 217}(2013), 1764-1772.
\bibitem{HV} J. Herzog and M. Vladoiu, {\it Monomial ideals with primary components given by powers of monomial prime ideals}, Electron. J. Comb., {\bf 21}(2014), P1.69.
\bibitem{JMS} M. Jafari, A. Mafi and H. Saremi, {\it sequentially Cohen-Macaulay matroidal ideals}, Filomat, {\bf 34}(2020), 4233-4244.
\bibitem{KM1} Sh. Karimi and A. Mafi, {\it On stability properties of powers of polymatroidal ideals}, Collect. Math., {\bf 70}(2019), 357-365.
\bibitem{KM} D. Kiani and S. Saeedi Madani, {The edge ideals of complete multipartite hypergraphs}, Comm. Algebra, {\bf 43}(2015), 3020-3032.
\bibitem{KMS} M. Koolani, A. Mafi and P. Soufivand, {\it An upper bound on stability of powers of matroidal ideals}, arXiv:2308.14019.
\bibitem{M} A. Mafi, {\it Ratliff-Rush ideal and reduction numbers}, Comm. Algebra, {\bf 46}(2018), 1272-1276.
\bibitem{MN} A. Mafi and D. Naderi, {\it A note on stability properties of powers of polymatroidal ideals}, Bull. Iranian. Math. Soc., {\bf 48}(2022), 3937-3945.
\bibitem{SM} H. Saremi and A. Mafi, {Unmixedness and arithmetic properties of matroidal ideals}, Arch. Math., {\bf 114}(2020), 299-304.
\bibitem{V}
R. H. Villarreal, {\it Monomial Algebras}, Monographs and Research Notes in Mathematics, Chapman and Hall/CRC, (2015).
\bibitem{V1}
M. Vladoiu, {\it Equidimensional and unmixed ideals of Veronese type}, Comm. Algebra, {\bf 36}(2008), 3378-3392.
\end{thebibliography}
\end{document}